\newif\ifJOURNAL\global\JOURNALfalse
\newif\ifHYPER\global\HYPERtrue
\newif\ifINTERNAL\global\INTERNALfalse
\documentclass[10pt,b5paper,fleqn]{article}
\usepackage[utf8]{inputenc}
\usepackage{amsmath,amsfonts,amssymb}
\usepackage[mathscr]{eucal}
\usepackage{color}
\usepackage{graphicx}
\usepackage{amsthm}

\ifHYPER
\usepackage{hyperref}
\definecolor{ks-green}{rgb}{0.0,0.7,0.0}
\definecolor{ks-red}{rgb}{0.7,0.0,0.0}
\definecolor{ks-blue}{rgb}{0.0,0.0,0.7}
\hypersetup{colorlinks=true,citecolor=ks-green,linkcolor=ks-red}
\fi

\numberwithin{equation}{section}
\theoremstyle{plain}
\newtheorem{theorem}[equation]{Theorem}

\newtheorem{proposition}[equation]{Proposition}
\newtheorem{algorithm}[equation]{Algorithm}
\theoremstyle{definition}
\newtheorem{definition}[equation]{Definition}
\newtheorem{Example}[equation]{Example}
\newtheorem{Remark}[equation]{Remark}

\newenvironment{remark}{\emph{Remark.}}{}

\newenvironment{notation}{\emph{Notation.}}{}

\newcommand{\lnum}[1]{\makebox[0.5\mathindent][r]{\textnormal{\footnotesize{#1}}}}
\newenvironment{algtest}{\bgroup\bfseries\upshape
  \tabbing
    \hspace*{\mathindent}\=
    \hspace*{1.5em}\=\hspace*{1.5em}\=%
    \hspace*{1.5em}\=\hspace*{1.5em}\= \kill \>}{%
  \endtabbing\egroup}

\def\revddots{\mathinner{\mkern1mu\raise1pt\vbox{\kern7pt\hbox{.}}\mkern2mu
  \raise4pt\hbox{.}\mkern2mu\raise7pt\hbox{.}\mkern1mu}}

\newcommand{\block}[1]{\underline{#1}}
\newcommand{\tsfrac}[2]{\textstyle{\frac{#1}{#2}}}
\newcommand{\tabstrut}{\rule[-0.5ex]{0pt}{2.8ex}}
\newcommand{\mthstrut}{\rule[-0.5ex]{0pt}{2.2ex}}
\newcommand\trp{^{\!\top}}
\newcommand\inv{^{-1}}
\newcommand\numN{\mathbb{N}}
\newcommand\numZ{\mathbb{Z}}
\newcommand\numQ{\mathbb{Q}}
\newcommand\numR{\mathbb{R}}
\newcommand\numC{\mathbb{C}}

\newcommand{\alphabet}[1]{\mathcal{#1}}
\newcommand{\freeALG}[2]{#1\langle #2\rangle}

\newcommand{\field}[1]{\mathbb{#1}}
\newcommand{\als}[1]{\mathcal{#1}}
\newcommand{\complexity}{\mathcal{O}}
\newcommand{\aclo}[1]{\overline{#1}}
\newcommand{\length}[1]{\vert #1 \vert}

\DeclareMathOperator{\rank}{rank}

\DeclareMathOperator{\linsp}{span}

\begin{document}
\title{Horner Systems: How to efficiently evaluate\\
       non-commutative polynomials (by matrices)}
\author{Konrad Schrempf%
  \footnote{Contact: math@versibilitas.at (Konrad Schrempf),
    \url{https://orcid.org/0000-0001-8509-009X},
    Uni\-ver\-si\-tät Wien, Fakultät für Mathematik,
    Oskar-Morgenstern-Platz~1, 1090 Wien;
    FH Ober\-öster\-reich, Forschungsgruppe ASiC, 
    Ringstraße~43a, 4600 Wels; Austria.
    }
    \hspace{0.2em}\href{https://orcid.org/0000-0001-8509-009X}{%
    \includegraphics[height=10pt]{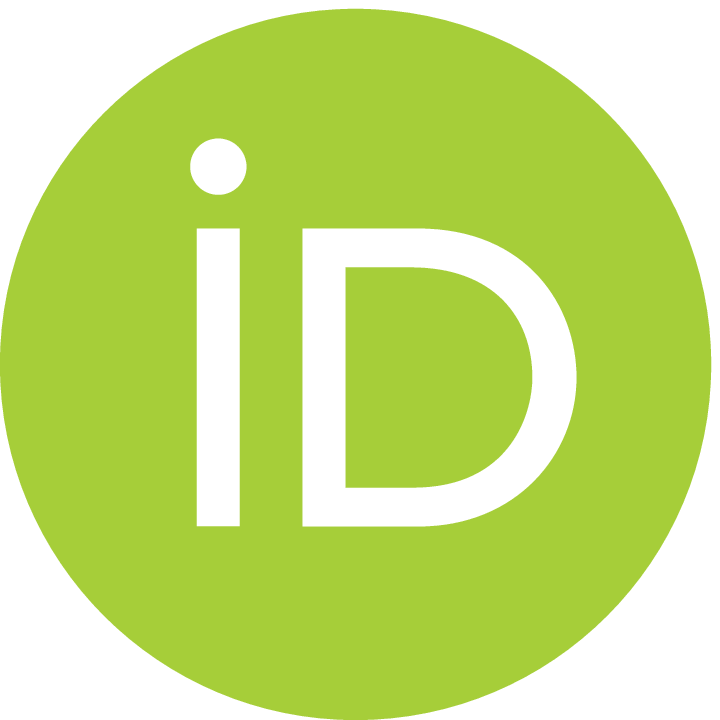}}
  }

\maketitle

\begin{abstract}
By viewing non-commutative polynomials, that is,
elements in \emph{free associative algebras}, in terms of
\emph{linear representations}, we generalize Horner's rule
to the non-commutative (multivariate) setting.
We introduce the concept of \emph{Horner systems}
(which has parallels to that of \emph{companion matrices}),
discuss their construction and
show how they enable the efficient evaluation of
non-commutative polynomials by matrices.
\end{abstract}

\medskip
\emph{Keywords and 2010 Mathematics Subject Classification.}
Horner's rule, free associative algebra, minimal linear representations,
admissible linear systems, matrix polynomials, companion matrix,
non-commutative factorization;
Primary 68W30;
Secondary 16Z05, 47A56

\section*{Introduction}

When we talk about the evaluation of non-commutative (nc) polynomials
by matrices, we actually take elements in the
\emph{free associative algebra}, aka ``algebra of non-commu\-ta\-tive polynomials'',
(over a \emph{commutative} field $\field{K}$,
e.g.~$\numQ$ or $\numC$,
and an alphabet $\alphabet{X}$ with $d$ letters) and view them as
functions on $d$-tuples of matrices (of appropriate sizes);
the non-commuting letters $x_1, x_2, \ldots, x_d$ 
(or $x,y,z$ for $d=3$) are
``placeholders'' where we plug in matrices $\bar{X}_1, \bar{X}_2, \ldots, \bar{X}_d$
(respectively $\bar{X},\bar{Y},\bar{Z}$).

Working symbolically with matrices (witout inverse) just means that we
add or multiply nc polynomials, that is, use the
ring operations in free associative algebras (over an
appropriate alphabet); usually in terms of (finite) formal
sums of words with coefficients in a \emph{commutative} field
$\field{K}$, 
for example $(x^2 + \frac{1}{2}xy) - (xy + 2y^2) = x^2 -\frac{1}{2}xy - 2y^2$
or $x\cdot(1-yx) = x-xyx$.

Another ---at a first glance much more complicated--- way to work
with nc polynomials is in terms of \emph{linear representations}
in the sense of Cohn and Reutenauer
\cite{Cohn1994a}
. Here a polynomial $p$ is written as $p = u A\inv v$
with $u\trp,v \in \field{K}^{n \times 1}$ and upper unitriangular
(with ones in the diagonal)
$n\times n$ matrix $A$
over \emph{linear} nc polynomials, for example
\begin{align*}
p = x - xyx &=
\begin{bmatrix}
1 & . & . & . 
\end{bmatrix}
\begin{bmatrix}
1 & -x & . & -x \\
. & 1 & y & . \\
. & . & 1 & -x \\
. & . & . & 1
\end{bmatrix}
\inv
\begin{bmatrix}
. \\ . \\ . \\ 1
\end{bmatrix}\\
&= 
\begin{bmatrix}
1 & . & . & . 
\end{bmatrix}
\begin{bmatrix}
1 & x & -xy & x-xyx \\
. & 1 & -y & -yx \\
. & . & 1 & x \\
. & . & . & 1
\end{bmatrix}
\begin{bmatrix}
. \\ . \\ . \\ 1
\end{bmatrix}
\end{align*}
(zero entries are replaced by lower dots to emphasize the structure).
The triple $\pi = (u,A,v)$ is called \emph{linear representation} of $p$,
the size of $A$ \emph{dimension}. If the dimension is the smallest possible
(for $p$), then $\pi$ is called \emph{minimal}.
Addition and multiplication can easily be formulated in terms
of linear representations
(discussed in detail in Section~\ref{sec:hs.faa}).
Furthermore, \emph{minimal} linear representations
can be used to factorize nc polynomials
(Section~\ref{sec:hs.mf}),
and from another point of view they are the natural generalization
of \emph{companion matrices}
(Section~\ref{sec:hs.hs}).

\begin{remark}
Here we restrict ourselves to the very special case of
nc polynomials. Linear representations in the sense of
Cohn and Reutenauer go far beyond, namely for elements
in a \emph{free field}
\cite{Amitsur1966a}
, that is, 
the \emph{universal field of fractions} of a
\emph{free associative algebra}
\cite[Chapter~7]{Cohn2006a}
. 
For a practical introduction see
\cite{Schrempf2018c}
, for the computation of the left gcd of two nc polynomials
\cite[Example~5.4]{Schrempf2018a2}
.
\end{remark}

\medskip
In other words: Linear representations are a powerful and
universal language in the context of (symbolic) non-commutative
rational expressions. For $u=[1,0,\ldots,0]$ we call
$\pi=(u,A,v)$ an \emph{admissible linear system} (ALS) for $p$
and write $\als{A}=\pi$ also
as $A s = v$. Then $p$ is the first component
of the (unique) solution vector $s$. Evaluating $p$ in terms of an
ALS by matrices is immediate: We start with $s_n = v_n$ and
compute $s_k$ for $k=n-1,\ldots,1$. Thus we do not need to
invert $A$ at all.

\begin{remark}
The term ``admissible'' means that the system matrix $A$
is invertible, that is, $A s = v$ admits a \emph{unique}
solution.
In our case $A$ is invertible over the
free associative algebra.
In general however, $A$ ``just'' needs to be invertible
over the \emph{free field}. Although this can be ensured by
a rather simple algebraic property it goes deep into the
heart of Cohn's theory and is very subtle and 
difficult to understand. (This is the actual reason to
restrict to the special case of nc polynomials.)
\end{remark}

\medskip
So, if we want to evaluate our polynomial $p = p(x,y)$
from before with $m\times m$ matrices $\bar{X},\bar{Y}$
we have
$s_4 = I_m$,
$s_3 = \bar{X} s_4 = \bar{X}$,
$s_2 = -\bar{Y} s_3 = -\bar{Y}\bar{X}$ and
$s_1 = \bar{X}s_2 + \bar{X}s_4 = -\bar{X}\bar{Y}\bar{X} + \bar{X}$.
Two matrix-matrix multiplications of complexity
$\complexity(m^3)$ are necessary,
one to compute $s_2$ and one to compute $p = s_1$.
The multiplication is the dominating part since the addition of matrices
has only complexity $\complexity(m^2)$.
In this case we did not gain anything by using linear representations
since plugging in $\bar{X},\bar{Y}$ directly into the
words $x$ and $xyx$ from $p$ would also ``cost''
two multiplications.

\begin{remark}
We only assume that the multiplication is the dominating part,
that is, its complexity is $\complexity(m^{2+\varepsilon})$
for $\varepsilon > 0$;
recall that Strassen's algorithm has $\complexity(m^{2.81})$
\cite{Strassen1969a}
.
And since we are interested in
\emph{practical} applications, (numerical) stability is
important. For details and references (including the complexity
of the matrix multiplication) we refer to 
\cite{Demmel2007a}
.
\end{remark}

\medskip
Now we take the polynomial
$p = 3 c y x b + 3 x b y x b + 2 c y x a x + c y b x b
   - c y a x b  - 2 x b y x a x + 4 x b y b x b  - 3 x b y a x b
   + 3 x a x y x b  - 3 b x b y x b + 6 a x b y x b + 2 x a x y x a x
   + x a x y b x b  - x a x y a x b  - 2 b x b y x a x  - b x b y b x b
   + b x b y a x b + 5 a x b y b x b - 4 a x b y a x b$
\cite[Section~8.2]{Camino2006a}
. (Here $a$, $b$ and $c$ can be viewed as matrix-valued parameters.)
If we want to evaluate $p = p(x,y,z;a,b,c)$ by matrices,
97~multiplications are necessary.
By rewriting $p$ as ``Sylvester mapping'' with respect to $y$,
that is $p = p_1 y q_1 + p_2 y q_2 + \ldots + p_k y q_k$,
the number of multiplications can be reduced to 28
\cite{Camino2006a}
. However, only $6+2+7=15$ multiplications
(left part, inner part, right part)
are necessary using the ``matrix-fac\-to\-ri\-zation''
\begin{displaymath}
\left(
\begin{bmatrix}
. & c 
\end{bmatrix}
+
\begin{bmatrix}
1+a & 1+b & x
\end{bmatrix}
\begin{bmatrix}
x & . & . \\
. & x & . \\
. & . & a 
\end{bmatrix}
\begin{bmatrix}
b & . \\
. & -b \\
. & x 
\end{bmatrix}
\right)
\begin{bmatrix}
y & . \\
. & y 
\end{bmatrix}
\begin{bmatrix}
6+5b-4a & . \\
3+b-a & 2x 
\end{bmatrix}
\begin{bmatrix}
. & x \\
a & . 
\end{bmatrix}
\begin{bmatrix}
x \\ b
\end{bmatrix}.
\end{displaymath}
How to find such factorizations is discussed
in Section~\ref{sec:hs.mf}
and summarized in Section~\ref{sec:hs.epi}.
Factorizations are important steps towards
``Horner systems'' which are ---roughly speaking---
the most sparse admissible linear systems (for a given
polynomial).

\begin{remark}
For Camino, Helton and Skelton 
\cite{Camino2006a}
\ the crucial point is to find the \emph{Sylvester index}
\cite{Konstantinov2000a}
, that is, the minimal number of ``summands''
(here it is~$k=2$ with respect to $y$),
to solve the generalized Sylvester equation.
Using three terms (instead of two) makes
a significant difference since no $\complexity(m^3)$
algorithm is known in the general case and the simple approach using
tensor product requires $\complexity(m^6)$
\cite{Simoncini2016a}
. See also \cite[Section~7.3]{Higham2008a}
.
\end{remark}

\medskip
In Section~\ref{sec:hs.faa} we give a brief introduction to
free associative algebras and set up the necessary formalism
to work with linear representations.
The main contribution is the concept of \emph{Horner systems}
(and bounds for the number of multiplications
in Proposition~\ref{pro:hs.complexity})
in Section~\ref{sec:hs.hs}.
From a practical point of view the \emph{minimization}
of linear representations (which we recall at the end of
Section~\ref{sec:hs.faa})
and the factorization into matrices in Section~\ref{sec:hs.mf}
are important since they are the major steps in the
construction of Horner systems.
And finally, in Section~\ref{sec:hs.epi},
we summarize how to construct Horner systems
and state some related literature.

\medskip
To get a first impression, one can start with Table~\ref{tab:hs.count}
(page~\pageref{tab:hs.count}). While the number of words ---and thus
the number of multiplications--- can grow exponentially, the number of
multiplications using Horner systems is at most quadratic with respect
to the \emph{rank} (Definition~\ref{def:hs.rep}), which is a good
``measure'' for the complexity of a nc polynomial.
(In the univariate case, the rank is just the degree plus one.
Notice however, that the rank function is \emph{not} a degree function.)

\medskip
\begin{notation}
The set of the natural numbers is denoted by $\numN = \{ 1,2,\ldots \}$.
Zero entries in matrices are usually replaced by (lower) dots
to emphasize the structure of the non-zero entries
unless they result from transformations where there
were possibly non-zero entries before.
We denote by $I_n$ the identity matrix
of size $n$ respectively $I$ if the size is clear from the context.
\end{notation}

\section{Free Associative Algebras}\label{sec:hs.faa}

After briefly introducing the ``algebra of nc polynomials''
and the notion of \emph{irreducible} polynomials (needed for the factorization),
we provide a detailed lead-in to the work with \emph{linear representations}
in the context of nc polynomials, companioned by examples.
At the end of this section we summarize the necessary setup
and present the algorithm for the \emph{minimization}.
For the factorization we refer
to Section~\ref{sec:hs.mf}

\medskip

Let $\field{K}$ be a \emph{commutative} field
(e.g.~$\numQ$, $\numR$ or $\numC$) and
$\alphabet{X} = \{ x_1, x_2, \ldots, x_d\}$ be a \emph{finite} (non-empty) alphabet.
The \emph{free monoid} $\alphabet{X}^*$ 
is the set of all \emph{finite words}
$x_{i_1} x_{i_2} \cdots x_{i_n}$ with $i_k \in \{ 1,2,\ldots, d \}$,
for example (for $\alphabet{X} = \{ x, y, z \}$),
\begin{displaymath}
\alphabet{X}^* = \{ 1, x, y, z, x^2, xy, xz, yx, y^2, yz, zx, zy, z^2, x^3, x^2y, \ldots \}.
\end{displaymath}
The multiplication on $\alphabet{X}^*$ is the \emph{concatenation},
that is,
$(x_{i_1} \cdots x_{i_m})\cdot (x_{j_1} \cdots x_{j_n})
= x_{i_1} \cdots x_{i_m} x_{j_1} \cdots x_{j_n}$,
with neutral element $1$, the \emph{empty word}.
The \emph{length} of a word $w=x_{i_1} x_{i_2} \cdots x_{i_m}$ is
(denoted by) $\length{w}=m$.
For an introduction see
\cite[Chapter~1]{Berstel2011a}
.

By $\freeALG{\field{K}}{\alphabet{X}}$ we denote the
\emph{free associative algebra} or \emph{free $\field{K}$-algebra}
(aka ``algebra of nc polynomials'').
Its elements can be uniquely expressed in the form
$\sum_{w \in \alphabet{X}^*} \kappa_w w$, $\kappa_w \in \field{K}$
(only finitely many $\kappa_w$ are non-zero),
that is, by \emph{finite} formal sums.
In the case of $\alphabet{X} = \{ x \}$,
the free associative algebra is just the polynomial ring $\field{K}[x]$.
Given two elements $p = \sum \kappa_w w$ and $q = \sum \lambda_w w$,
the \emph{sum} and the \emph{product} are given by
\begin{displaymath}
p + q = \sum_{w \in \alphabet{X}^*} (\kappa_w + \lambda_w) w
\quad\text{resp.}\quad
pq = \sum_{w \in \alphabet{X}^*} 
   \left( \sum_{uv=w} \kappa_u \lambda_v \right) w.
\end{displaymath}
A very rich resource on free associative algebras is
\cite{Cohn1974a}
. For their role in the theory of formal languages
we recommend 
\cite{Cohn1975b}
\ and 
\cite{Berstel2011a}
\ or
\cite{Salomaa1978a}
.

\medskip
For detailed algebraic discussions a lot of definitions (and
notations) are necessary. Therefore we formulate most as
a special case and refer to
\cite{Cohn1994a,Cohn1999a}
\ for linear representations and
\cite{Baeth2015a}
\ for the factorization
for further information and literature.
The \emph{factorization} in free associative algebras is a
natural generalization of that in the (ring of) integers $\numZ$.
However, in the non-commutative setting one needs to
distinguish between \emph{prime elements} (for divisibility) and
\emph{irreducible elements} or \emph{atoms} (for factorization).
(And the \emph{uniqueness} of a factorization into atoms needs
a generalization
\cite{Cohn1963b}
.)  The \emph{number} of atoms is unique,
for example $x - xyx = x(1-yx) = (1-xy)x$.

\begin{definition}[Irreducible Polynomials]\label{def:hs.irrpoly}
A (non-trivial) polynomial $p \in \freeALG{\field{K}}{\alphabet{X}} \setminus \field{K}$,
that is, a non-zero non-invertible element,
is called an \emph{atom} (or \emph{irreducible})
if $p = q_1 q_2$ implies that either $q_1 \in \field{K}$ or $q_2 \in \field{K}$,
that is, one of the factors is \emph{invertible}.
(Invertible elements are also called \emph{units}.)
\end{definition}

Now we go over to \emph{linear representations} of
elements in free associative algebras and formulate the
ring operations (sum and product) and
the factorization on that level.
There are two main issues we need to take care of:
\begin{itemize}
\item Does every polynomial admit a linear representation?
\item And, how can we construct \emph{minimal} linear representations?
\end{itemize}
Both can be addressed in a constructive way. We start
with ``minimal monomials'' (Proposition~\ref{pro:hs.minmon}),
add or multiply them (Proposition~\ref{pro:hs.ratop})
and minimize (Algorithm~\ref{alg:hs.minals}).
We illustrate these steps using $p=x$ and $q=1-yx$
with ``manual'' minimization to avoid a lot of technical details
(necessary for an implementation in computer algebra systems).

\begin{definition}[Linear Representations, Dimension, Rank
\cite{Cohn1994a,Cohn1999a}
]\label{def:hs.rep}
Let $p \in \freeALG{\field{K}}{\alphabet{X}}$.
A \emph{linear representation} of $p$ is a triple $\pi = (u,A,v)$ with
$u\trp,v \in \field{K}^{n \times 1}$
(for some $n \in \numN$)
and an over $\freeALG{\field{K}}{\alphabet{X}}$ invertible
$n \times n$ matrix $A = (a_{ij})$ with entries
$a_{ij} = \kappa_{ij}^{(0)} + \kappa_{ij}^{(1)} x_1 + \ldots + \kappa_{ij}^{(d)} x_d$,
$\kappa_{ij}^{(\ell)} \in \field{K}$,
and $p = u A\inv v$.
The \emph{dimension} of $\pi$ is $\dim \, (u,A,v) = n$.
It is called \emph{minimal} if $A$ has the smallest possible dimension
among all linear representations of $p$.
The ``empty'' representation $\pi = (,,)$ is
the minimal one of $0 \in \freeALG{\field{K}}{\alphabet{X}}$
with $\dim \pi = 0$.
Let $p \in \freeALG{\field{K}}{\alphabet{X}}$
and $\pi$ be a \emph{minimal} linear representation of $p$.
Then the \emph{rank} of $p$ is
defined as $\rank p = \dim \pi$.
\end{definition}

\begin{definition}[Left and Right Families
\cite{Cohn1994a}
]\label{def:hs.family}
Let $\pi=(u,A,v)$ be a linear representation of
$p \in \freeALG{\field{K}}{\alphabet{X}}$
of dimension $n$.
The families $( s_1, s_2, \ldots, s_n )\subseteq \freeALG{\field{K}}{\alphabet{X}}$
with $s_i = (A\inv v)_i$
and $( t_1, t_2, \ldots, t_n )\subseteq \freeALG{\field{K}}{\alphabet{X}}$
with $t_j = (u A\inv)_j$
are called \emph{left family} and \emph{right family} respectively.
$L(\pi) = \linsp \{ s_1, s_2, \ldots, s_n \}$ and
$R(\pi) = \linsp \{ t_1, t_2, \ldots, t_n \}$
denote their linear spans (over $\field{K}$).
\end{definition}

\begin{proposition}[%
\protect{\cite[Proposition~4.7]{Cohn1994a}
}]\label{pro:hs.cohn94.47}
A representation $\pi=(u,A,v)$ of an element $p \in \freeALG{\field{K}}{\alphabet{X}}$
is minimal if and only if both, the left family
and the right family, are $\field{K}$-linearly independent.
In this case, $L(\pi)$ and $R(\pi)$ depend only on $p$.
\end{proposition}

\begin{definition}[Admissible Linear Systems and Transformations
\cite{Schrempf2017a9}
]\label{def:hs.als}
A linear representation $\als{A} = (u,A,v)$ of $p \in \freeALG{\field{K}}{\alphabet{X}}$
is called \emph{admissible linear system} (ALS) for $p$,
written also as $A s = v$,
if $u=e_1=[1,0,\ldots,0]$.
The element $p$ is then the first component
of the (unique) solution vector $s$.
Given a linear representation $\als{A} = (u,A,v)$
of dimension $n$ of $p \in \freeALG{\field{K}}{\alphabet{X}}$
and invertible matrices $P,Q \in \field{K}^{n\times n}$,
the transformed $P\als{A}Q = (uQ, PAQ, Pv)$ is
again a linear representation (of $p$).
If $\als{A}$ is an ALS,
the transformation $(P,Q)$ is called
\emph{admissible} if the first row of $Q$ is $e_1 = [1,0,\ldots,0]$.
\end{definition}

\begin{remark}
The left family $(A\inv v)_i$ (respectively the right family $(u A\inv)_j$)
and the solution vector $s$ of $As = v$ (respectively $t$ of $u = tA$)
are used synonymously.
\end{remark}

\begin{remark}
Transformations can be done by elementary row and column operations.
However, we are not allowed to scale the first column or add a multiple of
it to other columns (because this would change the first entry in
the left family).
\end{remark}

\begin{Example}\label{ex:hs.defals}
A \emph{minimal} admissible linear system
for $p = x$ is given by 
\begin{displaymath}
\als{A}_p = 
\left(
\begin{bmatrix}
1 & . 
\end{bmatrix},
\begin{bmatrix}
1 & -x \\
. & 1
\end{bmatrix},
\begin{bmatrix}
. \\ 1
\end{bmatrix}
\right).
\end{displaymath}
The left family is $s = (x, 1)$,
the right family is $t = (1,x)$.

\noindent
A \emph{minimal} admissible linear system
for $q = 1-yx$ is given by
\begin{displaymath}
\als{A}_q = 
\left(
\begin{bmatrix}
1 & . & . 
\end{bmatrix},
\begin{bmatrix}
1 & y & -1 \\
. & 1 & -x \\
. & . & 1
\end{bmatrix},
\begin{bmatrix}
. \\ . \\ 1
\end{bmatrix}
\right).
\end{displaymath}
The left family is $s = (1,x,1-yx)$,
the right family is $t = (1, -y, 1-yx)$.
\end{Example}

\begin{definition}[Polynomial ALS and Transformation
\protect{\cite[Definition~24]{Schrempf2017b9}
}]\label{def:hs.pals}
An ALS $\als{A} = (u,A,v)$ of dimension $n$
with system matrix $A = (a_{ij})$
for a non-zero element $p \in \freeALG{\field{K}}{\alphabet{X}} \setminus \{ 0 \}$ 
is called \emph{polynomial} ALS, if
\begin{itemize}
\item[(1)] $v = [0,\ldots,0,\lambda]\trp$ for some $\lambda \in\field{K}$ and
\item[(2)] $a_{ii}=1$ for $i=1,2,\ldots, n$ and $a_{ij}=0$ for $i>j$,
  that is, $A$ is upper triangular.
\end{itemize}
A polynomial ALS is also written as $\als{A} = (1,A,\lambda)$
with $1,\lambda \in \field{K}$.
An admissible transformation $(P,Q)$
for a polynomial ALS $\als{A}$
is called \emph{polynomial} if it has the form
\begin{displaymath}
(P,Q) = \left(
\begin{bmatrix}
1 & \alpha_{1,2} & \ldots & \alpha_{1,n-1} & \alpha_{1,n} \\
  & \ddots & \ddots & \vdots & \vdots \\
  &   & 1 & \alpha_{n-2,n-1} & \alpha_{n-2,n} \\
  &   &   & 1 & \alpha_{n-1,n} \\
  &   &   &   & 1
\end{bmatrix},
\begin{bmatrix}
1 & 0 & 0 & \ldots & 0 \\
  & 1 & \beta_{2,3} & \ldots & \beta_{2,n} \\
  &   & 1 & \ddots & \vdots  \\
  &   &   & \ddots & \beta_{n-1,n} \\
  &   &   &   & 1 \\
\end{bmatrix}
\right).
\end{displaymath}
If additionally $\alpha_{1,n} = \alpha_{2,n} = \ldots = \alpha_{n-1,n} = 0$
then $(P,Q)$ is called \emph{polynomial factorization transformation}.
Later we need also more general 
transformations~\eqref{eqn:hs.gft}.
\end{definition}

\begin{proposition}[Minimal Monomial
\protect{\cite[Proposition~4.1]{Schrempf2017a9}
}]\label{pro:hs.minmon}
Let $k \in \numN$ and $p= x_{i_1} x_{i_2} \cdots x_{i_k}$ be a monomial in
$\freeALG{\field{K}}{\alphabet{X}}$.
Then
\begin{displaymath}
\als{A} = \left(
\begin{bmatrix}
1 & 0  & \cdots & 0
\end{bmatrix},
\begin{bmatrix}
1 & -x_{i_1} \\
  & 1 & -x_{i_2} \\
  & & \ddots & \ddots \\
  & & & 1 & -x_{i_k} \\
  & & & & 1
\end{bmatrix},
\begin{bmatrix}
0 \\ 0 \\ \vdots \\ 0 \\ 1
\end{bmatrix}
\right)
\end{displaymath}
is a \emph{minimal} polynomial ALS of dimension $\dim \als{A} = k+1$.
\end{proposition}

\begin{proposition}[Rational Operations
\cite{Cohn1999a}
]\label{pro:hs.ratop}
Let $0\neq p,q \in \freeALG{\field{K}}{\alphabet{X}}$ be given by the
admissible linear systems $\als{A}_p = (u_p, A_p, v_p)$
and $\als{A}_q = (u_q, A_q, v_q)$ respectively.
Then an ALS for the sum $p + q$ is given by
\begin{displaymath}
\als{A}_p + \als{A}_q =
\left(
\begin{bmatrix}
u_p & . 
\end{bmatrix},
\begin{bmatrix}
A_p & -A_p u_p\trp u_q \\
. & A_q
\end{bmatrix}, 
\begin{bmatrix} v_p \\ v_q \end{bmatrix}
\right).
\end{displaymath}
And an ALS for the product $fg$ is given by
\begin{displaymath}
\als{A}_p \cdot \als{A}_q =
\left(
\begin{bmatrix}
u_p & . 
\end{bmatrix},
\begin{bmatrix}
A_p & -v_p u_q \\
. & A_q
\end{bmatrix},
\begin{bmatrix}
. \\ v_q
\end{bmatrix}
\right).
\end{displaymath}
\end{proposition}

\medskip
\begin{Example}
An ALS for $h_1 = p + q$ from Example~\ref{ex:hs.defals} is
\begin{displaymath}
\begin{bmatrix}
1 & -x & -1 & . & . \\
. & 1 & . & . & . \\
. & . & 1 & y & -1 \\
. & . & . & 1 & -x \\
. & . & . & . & 1
\end{bmatrix}
s = 
\begin{bmatrix}
. \\ 1 \\ . \\ . \\ 1
\end{bmatrix},
\quad
s = 
\begin{bmatrix}
1 + x - yx \\ 1 \\ 1-yx \\ x \\ 1
\end{bmatrix}.
\end{displaymath}
If we add row~3 to row~1 we get
\begin{displaymath}
\begin{bmatrix}
1 & -x & 0 & y & -1 \\
. & 1 & . & . & . \\
. & . & 1 & y & -1 \\
. & . & . & 1 & -x \\
. & . & . & . & 1
\end{bmatrix}
s = 
\begin{bmatrix}
. \\ 1 \\ . \\ . \\ 1
\end{bmatrix},
\quad
s = 
\begin{bmatrix}
1 + x - yx \\ 1 \\ 1-yx \\ x \\ 1
\end{bmatrix}
\end{displaymath}
and can remove row/column~3 because the corresponding
column equation reads $t_3 = 0$ (recall that $u_j=0$
for $j\ge 2$ in an ALS):
\begin{displaymath}
\begin{bmatrix}
1 & -x & y & -1 \\
. & 1 & . & . \\
. & . & 1 & -x \\
. & . & . & 1
\end{bmatrix}
s = 
\begin{bmatrix}
. \\ 1 \\ . \\ 1
\end{bmatrix},
\quad
s = 
\begin{bmatrix}
1 + x - yx \\ 1 \\ x \\ 1
\end{bmatrix}.
\end{displaymath}
Now we can subtract row~4 from row~2 and add column~2 to column~4
(which results in subtracting $s_4$ from $s_2$):
\begin{displaymath}
\begin{bmatrix}
1 & -x & y & -1-x \\
. & 1 & . & 0 \\
. & . & 1 & -x \\
. & . & . & 1
\end{bmatrix}
s = 
\begin{bmatrix}
. \\ 0 \\ . \\ 1
\end{bmatrix},
\quad
s = 
\begin{bmatrix}
1 + x - yx \\ 0 \\ x \\ 1
\end{bmatrix}.
\end{displaymath}
Removing row/column~2 yields a \emph{minimal} ALS
(of dimension~3). Thus $\rank(p+q) = 3$.
An ALS for $h_2 = pq$ is
\begin{displaymath}
\begin{bmatrix}
1 & -x & . & . & . \\
. & 1 & -1 & . & . \\
. & . & 1 & y & -1 \\
. & . & . & 1 & -x \\
. & . & . & . & 1
\end{bmatrix}
s =
\begin{bmatrix}
. \\ . \\ . \\ . \\ 1
\end{bmatrix},
\quad
s =
\begin{bmatrix}
x(1-yx) \\ 1-yx \\ 1-yx \\ x \\ 1
\end{bmatrix}.
\end{displaymath}
Since $p$ and $q$ are given by \emph{minimal} admissible
linear systems, there is exactly one minimization step
possible. Here we add row~3 to row~2 and remove
row/column~3:
\begin{displaymath}
\begin{bmatrix}
1 & -x & 0 & 0 \\
. & 1 & y & -1 \\
. & . & 1 & -x \\
. & . & . & 1
\end{bmatrix}
s =
\begin{bmatrix}
. \\ . \\ . \\ 1
\end{bmatrix},
\quad
s =
\begin{bmatrix}
x(1-yx) \\ 1-yx \\ x \\ 1
\end{bmatrix}.
\end{displaymath}
Notice the upper right block of zeros of size $1\times 2$
in the system matrix. This is what we need for the factorization
later (Theorem~\ref{thr:hs.factorization}).
\end{Example}

The ``left'' (row) and ``right'' (column) minimization steps are rather simple.
However, to ensure minimality, we need to do that systematically.
For further details we refer to 
\cite[Section~2]{Schrempf2017b9}
.
To formulate the algorithm we need to decompose the polynomial ALS
$\als{A} = (u,A,v)$ of dimension $n \ge 2$
with respect to some row/column~$k$:
\begin{displaymath}
\als{A}^{[k]} = \left(
\begin{bmatrix}
u_{\block{1}} & . & . 
\end{bmatrix},
\begin{bmatrix}
A_{1,1} & A_{1,2} & A_{1,3} \\
. & 1 & A_{2,3} \\
. & . & A_{3,3}
\end{bmatrix},
\begin{bmatrix}
v_{\block{1}} \\ v_{\block{2}} \\ v_{\block{3}}
\end{bmatrix}
\right).
\end{displaymath}
(To avoid confusion, we use underlined subscripts to denote blocks
in vectors.)\\
By $\als{A}^{[-k]}$ we denote the ALS $\als{A}^{[k]}$
without row/column~$k$ (of dimension $n-1$):
\begin{displaymath}
\als{A}^{[-k]} = \left(
\begin{bmatrix}
u_{\block{1}} & . 
\end{bmatrix},
\begin{bmatrix}
A_{1,1} & A_{1,3} \\
. & A_{3,3}
\end{bmatrix},
\begin{bmatrix}
v_{\block{1}} \\ v_{\block{3}}
\end{bmatrix}
\right).
\end{displaymath}
Removing row/column~$k$ is only ``admissible''
if either $A_{1,2}=0$ or $A_{2,3} =0$ (and $v_{\block{2}}=0$).
For row minimization steps we use the transformation
\begin{displaymath}
\bigl(P(T), Q(U) \bigr) = \left(
\begin{bmatrix}
I_{k-1} & . & . \\
. & 1 & T \\
. & . & I_{n-k}
\end{bmatrix},
\begin{bmatrix}
I_{k-1} & . & . \\
. & 1 & U \\
. & . & I_{n-k}
\end{bmatrix}
\right),
\end{displaymath}
for column minimization steps we use
\begin{displaymath}
\bigl(P(T), Q(U) \bigr) = \left(
\begin{bmatrix}
I_{k-1} & T & . \\
. & 1 & . \\
. & . & I_{n-k}
\end{bmatrix},
\begin{bmatrix}
I_{k-1} & U & . \\
. & 1 & . \\
. & . & I_{n-k}
\end{bmatrix}
\right).
\end{displaymath}

\begin{definition}[Minimization Equations, Transformations
\protect{\cite[Definition~31]{Schrempf2017b9}
}]\label{def:hs.meqn}
Let $\als{A} = (u,A,v)$ be a polynomial ALS of dimension
$n \ge 2$.
For $k = \{ 1,2,\ldots, n-1 \}$ the equations 
$U + A_{2,3} + T A_{3,3} = 0$ and $v_{\block{2}} + T v_{\block{3}} = 0$,
with respect to the block decomposition $\als{A}^{[k]}$
are called \emph{left minimization equations},
denoted by $\mathcal{L}_k = \mathcal{L}_k(\als{A})$.
A solution by the row block pair $(T,U)$ is denoted by
$\mathcal{L}_k(T,U) = 0$, 
the corresponding transformation
$(P,Q) =\bigl(P(T), Q(U) \bigr)$
is called \emph{left minimization transformation}.
For $k = \{ 2,3,\ldots, n \}$ the equations
$A_{1,1} U + A_{1,2} + T = 0$, 
with respect to the block decomposition $\als{A}^{[k]}$
are called \emph{right minimization equations},
denoted by $\mathcal{R}_k = \mathcal{R}_k(\als{A})$.
A solution by the column block pair $(T,U)$ is denoted by
$\mathcal{R}_k(T,U) = 0$,
the corresponding transformation
is called \emph{right minimization transformation}.
\end{definition}

\begin{algorithm}[Minimizing a polynomial ALS
\protect{\cite[Algorithm~32]{Schrempf2017b9}
}]\label{alg:hs.minals}
\ \\
Input: $\als{A} = (u,A,v)$ polynomial ALS
  of dimension $n \ge 2$ (for some polynomial $p$).\\
Output: $\als{A}' = (,,)$ if $p=0$ or
        a minimal polynomial ALS $\als{A}' = (u',A',v')$ if $p \neq 0$.
        
\begin{algtest}
\hbox{}\\[-3ex]
\lnum{1:}\>$k := 2$ \\
\lnum{2:}\>while $k \le \dim \als{A}$ do \\
\lnum{3:}\>\>$n := \dim(\als{A})$ \\
\lnum{4:}\>\>$k' := n  +1 - k$ \\
\lnum{  }\>\>\textnormal{Is the left subfamily
  \raisebox{0pt}[0pt][0pt]{%
    $(s_{k'},\overbrace{\mthstrut s_{k'+1},\ldots, s_{n}}^{\text{lin.~indep.}})$}
    $\field{K}$-linearly dependent?} \\
\lnum{5:}\>\>if $\exists\, T,U \in \field{K}^{1 \times (k-1)}
  \textnormal{ admissible}: \mathcal{L}_{k'}(\als{A})=\mathcal{L}_{k'}(T,U)=0$ then \\
\lnum{6:}\>\>\>if $k' = 1$ then \\
\lnum{7:}\>\>\>\>return $(,,)$ \\
\lnum{  }\>\>\>endif \\
\lnum{8:}\>\>\>\raisebox{0pt}[0pt][0pt]{%
      $\als{A} := \bigl(P(T) \als{A}\, Q(U)\bigr)\mthstrut^{[-k']}$} \\
\lnum{9:}\>\>\>if $k > \max \bigl\{ 2, \frac{n+1}{2} \bigr\}$ then \\
\lnum{10:}\>\>\>\>$k := k-1$ \\
\lnum{   }\>\>\>endif \\
\lnum{11:}\>\>\>continue \\
\lnum{   }\>\>endif \\
\lnum{12--17:}\>\>\textnormal{(for alignment)}\\
\lnum{   }\>\>\textnormal{Is the right subfamily
      \raisebox{0pt}[0pt][0pt]{%
            $(\overbrace{\mthstrut t_1, \ldots,t_{k-1}}^{\text{lin.~indep.}}, t_k)$}
          $\field{K}$-linearly dependent?} \\
\lnum{18:}\>\>if $\exists\, T,U \in \field{K}^{(k-1) \times 1}
    \textnormal{ admissible} : 
     \mathcal{R}_k(\als{A}) = \mathcal{R}_k(T,U)=0$ then \\
\lnum{19:}\>\>\>$\als{A} := \bigl(P(T) \als{A}\, Q(U) \bigr)\mthstrut^{[-k]}$ \\
\lnum{20:}\>\>\>if $k > \max \bigl\{ 2, \frac{n+1}{2} \bigr\}$ then \\
\lnum{21:}\>\>\>\>$k := k-1$ \\
\lnum{   }\>\>\>endif \\
\lnum{22:}\>\>\>continue \\
\lnum{   }\>\>endif \\
\lnum{23:}\>\>$k := k+1$ \\
\lnum{   }\>done \\
\lnum{24:}\>return $P\als{A},$ 
        \textnormal{with $P$, such that $Pv = [0,\ldots,0,\lambda]\trp$}
\end{algtest}
\end{algorithm}

\medskip
\begin{remark}
The line numbering is with respect to the
general algorithm \cite[Algorithm~4.14]{Schrempf2018a2}
. Polynomial admissible linear systems,
called ``pre-standard'' in 
\cite{Schrempf2017b9}
,
are a special case of
\emph{refined} admissible linear systems because their
diagonal blocks are as small as possible, namely $1\times 1$.
The lines 12--15 in 
\cite[Algorithm~32]{Schrempf2017b9}
\ are not even necessary since this special case
is detected in the following part of the algorithm
(the right family is linearly dependent).
\end{remark}

\begin{Remark}
If a linear representation $\pi=(u,A,v)$, say of dimension~$n$,
of some element $p \in \freeALG{\field{K}}{\alphabet{X}}$
is \emph{not} in the form of a polynomial ALS, there exists
\emph{invertible} matrices $P,Q\in\field{K}^{n \times n}$ such
that $\als{A} = P \pi Q$ has this form.
In general, $\pi$ needs to be \emph{minimal}. Then
\cite[Proposition~2.1]{Cohn1999a}
\ ensures the existence of an upper unitriangular linear
representation and 
\cite[Theorem~1.4]{Cohn1999a}
\ implies the existence of such $P$ and $Q$,
which are usually difficult to find.
In our situation it is much simpler: We get the existence
of a \emph{minimal} polynomial ALS (for each element in the
free associative algebra) directly by construction.
\end{Remark}

\section{Horner Systems}\label{sec:hs.hs}

Given $p \in \freeALG{\field{K}}{\alphabet{X}} \setminus\field{K}$
by a \emph{polynomial} ALS $\als{A} = (1,A,\lambda)$,
say of dimension $n$, it is almost straight forward
to generalize the idea of Horner's rule to the non-com\-mu\-ta\-ti\-ve
setting once we recall how to evaluate $p$ by a $d$-tuple
of $m\times m$ matrices $\bar{X}_1, \bar{X}_2, \ldots, \bar{X}_d$
for the letters $x_i \in \alphabet{X}$
(abusing the notation for the left and the right family):
\begin{itemize}
\item Starting with $s_n = I_m$, we compute (rowwise)
  $s_{n-1}$ to $s_1 = p$.
\item Or, starting with $t_1 = I_m$, we compute (columnwise)
  $t_2$ to $t_n = \frac{1}{\lambda} p$.
\end{itemize}
Although we will see later (in Remark~\ref{rem:hs.nonminimality})
that \emph{minimal} admissible linear systems are not necessarily
optimal with respect to the number of multiplications (for the evaluation),
minimization (Algorithm~\ref{alg:hs.minals})
is \emph{the} major step towards \emph{Horner systems}
(Definition~\ref{def:hs.hs}). This becomes visible in particular
in Table~\ref{tab:hs.count}. Minimality plays also a crucial role for
the \emph{factorization} of a polynomial into a product of atoms
(irreducible elements), or an atom into a product of matrices,
and thus for creating (upper right) blocks of zeros in the system
matrix $A$ (if possible). For details we refer to Section~\ref{sec:hs.mf}.

\begin{Remark}
Finding the ``most sparse'' polynomial ALS can be very
difficult in general because \emph{non}-linear systems
of equations need to be solved,
similarly to
\cite[Proposition~42]{Schrempf2017b9}
. So the minimization is
rather cheap since it can be done with complexity
$\complexity(d n^4)$. For details we refer to
\cite[Remark~33]{Schrempf2017b9}
.
Fortunately one can also try \emph{linear} (algebraic)
techniques to ``break'' huge polynomials into smaller
factors
\cite[Remark~5.8]{Schrempf2018a2}
.
\end{Remark}

\begin{table}
\begin{center}
\begin{tabular}{rr|rrr|rrr|}
$k$ & rank & \# terms & \# mult.& $N(p_k)$ & \# terms & \# mult.& $N(q_k)$ \\\hline\tabstrut
0 &  1 &     1 &      0 & 0 &      1 &       0 &  0 \\
1 &  2 &     3 &      0 & 0 &      3 &       0 &  0 \\
2 &  3 &     9 &      9 & 1 &     12 &       9 &  1 \\
3 &  4 &    27 &     54 & 2 &     48 &      72 &  3 \\
4 &  5 &    81 &    243 & 3 &    192 &     432 &  6 \\
5 &  6 &   243 &    972 & 4 &    768 &    2304 & 10 \\
6 &  7 &   729 &   3645 & 5 &   3072 &   11520 & 15 \\
7 &  8 &  2187 &  13122 & 6 &  12288 &   73728 & 21 \\
8 &  9 &  6561 &  45927 & 7 &  49152 &  344064 & 28 \\
9 & 10 & 19683 & 157464 & 8 & 196608 & 1572864 & 36 \\\hline\tabstrut
$k$ & $k+1$ & $3^k$ & $(k-1) 3^k$ & $k-1$ & $3\cdot 4^{k-1}$ &
  Rem.~\ref{rem:hs.count} & $\frac{k}{2}(k-1)$
\end{tabular}
\end{center}
\caption{Number of multiplications for the evaluation of $p_k = (x+y+z)^k$
  (column~4 resp.~5)
  and $q_k = (x_1+y_1+z_1)q_{k-1}+ ... + (x_k+y_k+z_k)q_1$
  (column~7 resp.~8)
  as (finite) formal sum respectively
  minimal polynomial ALS.
  See also Remark~\ref{rem:hs.count}.
  }
\label{tab:hs.count}
\end{table}

Since there are close connections to \emph{companion matrices}
(for the univariate case)
we recall some basics and start with
\emph{companion systems} (Definition~\ref{def:hs.cs})
to construct \emph{minimal} polynomial admissible linear systems.

A \emph{univariate} polynomial 
$p(x) = a_0 + a_1 x + \ldots + a_{n-1} x^{n-1} + x^n
  \in \field{K}[x] = \freeALG{\field{K}}{\{x\}}$
can be expressed as the \emph{characteristic polynomial}
of its \emph{companion matrix} $L = L(p)$, that is,
$p(x) = \det(xI - L)$
\cite[Section~VI.6]{Gantmacher1966a}
,
\begin{displaymath}
p(x) = \det(\underbrace{xI - L}_{=: C(p)}) = \det
\begin{bmatrix}
x & 0 & \ldots & 0 & a_0 \\
-1 & x & \ddots & \vdots & a_1 \\
  & \ddots & \ddots & 0 & \vdots \\
  &  & -1 & x & a_{n-2} \\
  &  &   & -1 & x+a_{n-1}
\end{bmatrix}.
\end{displaymath}
In \cite[Section~8.1]{Cohn1995a}
, $\tilde{C}(p) = xI - L(p)\trp$ is also called \emph{companion matrix}.
Viewing $C(p)$ as \emph{linear matrix pencil}
$C(p) = C_0 \otimes 1 + C_x \otimes x$
generalizes nicely to nc polynomials:
$C(p)$ is ---modulo sign--- just the upper right $(n-1) \times (n-1)$ block
of the system matrix of the (minimal) \emph{right} companion system
$\mathcal{C}_p = (u,A,v) = (1,A,1)$ of dimension~$n$
(Definition~\ref{def:hs.cs}).
Evaluating $p$ in the special case of $q_i = x$ starting from the bottom right
in this \emph{minimal} ALS yields directly
\emph{Horner's rule}.
Notice that here $a_n = 1$, thus $n-1=\rank(p)-2$ multiplications
are needed.

\begin{remark}
Notice that the system matrix $A$ in
Definition~\ref{def:hs.rep} could be also written using
the tensor product $A = A_0 \otimes 1 + A_1 \otimes x_1 + \ldots + A_d \otimes x_d$,
$A_i \in \field{K}^{n \times n}$,
which reduces to the Kronecker tensor product when we plug
in $m \times m$ matrices $\bar{X}_1, \ldots, \bar{X}_d$:
$\bar{A} = A_0 \otimes I_m  + A_1 \otimes \bar{X}_1 + \ldots + A_d \otimes \bar{X}_d
\in \field{K}^{mn \times mn}$.
\end{remark}

\begin{remark}
In \cite[Section~11.1]{Bart2008a}
\ the companion matrix $L(p)$ is called \emph{second companion},
its transpose $L(p)\trp$ \emph{first companion} (matrix).
\end{remark}

\begin{definition}[Companion Systems
\protect{\cite[Definition~46]{Schrempf2017b9}
}]\label{def:hs.cs}
For $i=1,2,\ldots, m$ let $q_i \in \freeALG{\field{K}}{\alphabet{X}}$
with $\rank{q_i} = 2$ and $a_i \in \field{K}$.
For a polynomial $p \in \freeALG{\field{K}}{\alphabet{X}}$
of the form\\
$p = q_m q_{m-1} \cdots q_1 + a_{m-1} q_{m-1} \cdots q_1 + \ldots + a_2 q_2 q_1
    + a_1 q_1 + a_0$
the polynomial ALS
\begin{equation}\label{eqn:hs.lcs}
\begin{bmatrix}
1 & -q_m -a_{m-1} & -a_{m-2} & \ldots & -a_1 & -a_0 \\
  & 1 & -q_{m-1} & 0 & \ldots & 0 \\
  &   & \ddots & \ddots & \ddots & \vdots \\
  &   &   &  1 & -q_2 & 0 \\
  &   &   &  & 1 & -q_1 \\
  &   &   &  & & 1
\end{bmatrix}
s = 
\begin{bmatrix}
0 \\ 0 \\ \vdots \\ 0 \\ 0 \\ 1
\end{bmatrix}
\end{equation}
is called \emph{left companion system}.
And for a polynomial $p \in \freeALG{\field{K}}{\alphabet{X}}$ of the form\\
$p = a_0 + a_1 q_1 + a_2 q_1 q_2 + \ldots + a_{m-1} q_1 q_2 \cdots q_{m-1}
    + q_1 q_2 \cdots q_m$
the polynomial ALS
\begin{equation}\label{eqn:hs.rcs}
\begin{bmatrix}
1 & -q_1 & 0  & \ldots & 0 & -a_0 \\
  & 1 & -q_2 & \ddots & \vdots & -a_1 \\
  &   & \ddots & \ddots & 0 & \vdots \\
  &   &   &  1 & -q_{m-1} & -a_{m-2} \\
  &   &   &  & 1 & -q_m -a_{m-1} \\
  &   &   &  & & 1
\end{bmatrix}
s = 
\begin{bmatrix}
0 \\ 0 \\ \vdots \\ 0 \\ 0 \\ 1
\end{bmatrix}
\end{equation}
is called \emph{right companion system}.
\end{definition}

\begin{Example}[%
\protect{\cite[Example~50]{Schrempf2017b9}
}]
The left companion system of\\
$p(x) = x^3 - 10 x^2 + 31 x - 30$ is
\begin{displaymath}
\begin{bmatrix}
1 & -x+10 & -31 & 30 \\
. & 1 & -x & . \\
. & . & 1 & -x \\
. & . & . & 1
\end{bmatrix}
s =
\begin{bmatrix}
. \\ . \\ . \\ 1
\end{bmatrix},
\quad
s =
\begin{bmatrix}
p(x) \\
x^2 \\
x \\
1
\end{bmatrix}.
\end{displaymath}
\end{Example}

\begin{Remark}
One can view the algorithm (in the univariate case) in
\cite{Tajima2014a}
\ as taking the \emph{left} companion system
\eqref{eqn:hs.lcs}
and evaluate the matrix powers in the left family
(Definition~\ref{def:hs.family})
$s = (p, x^{n-1}, \ldots, x^2, x, 1)$
efficiently, for example, $x^4 = x^2 \cdot x^2$.
Notice, that in this case the coefficients $a_i \in \field{K}$ are
assumed to be scalar. Matrix valued coefficients (or parameters)
can easily be treated by an augmented alphabet,
here $\tilde{\alphabet{X}} = \alphabet{X} \cup \{ a_0, a_1, \ldots, a_{n-1} \}$.
\end{Remark}

Here we consider the general case and assume that
our alphabet $\alphabet{X}$ contains the
matrix valued parameters (mainly $a$, $b$ and $c$).
Although we can evaluate a polynomial
with matrices of appropriate sizes,
we typically plug in $m \times m$ matrices and measure
the ``evaluation complexity'' as the minimal number of
matrix-matrix multiplications
with respect to a \emph{polynomial} ALS.

\begin{remark}
Recall that we assume only that the multiplication is the dominating part,
that is, its complexity is $\complexity(m^{2+\varepsilon})$
for $\varepsilon > 0$.
\end{remark}

\begin{definition}[Left/Right/Minimal Number of Multiplications]
Let $p \in \freeALG{\field{K}}{\alphabet{X}}$ be given by
the polynomial ALS $\als{A}=(1,A,\lambda)$ of dimension $n\ge 2$.
The minimal number of \emph{non-scalar} entries in the upper left
(respectively lower right)
$(n-1)\times (n-1)$ block of $A$ is called
\emph{left} (respectively \emph{right})
\emph{number of multiplications}, written as $N_s(\als{A})$
(respectively $N_t(\als{A})$).
The number of multiplications (of a polynomial ALS) is
denoted by $N(\als{A}) = \min \{ N_s(\als{A}), N_t(\als{A}) \}$.
If $N(\als{A})\le N(\als{B})$ for all
polynomial admissible linear systems $\als{B}$ for $p$,
we write $N(p) = N(\als{A})$.
\end{definition}

\begin{definition}[Horner System]
\label{def:hs.hs}
Let $p \in \freeALG{\field{K}}{\alphabet{X}}$.
A polynomial ALS $\als{A}$ for $p$ is called \emph{Horner system}
if $N(\als{A}) = N(p)$.
\end{definition}

\begin{Remark}\label{rem:hs.nonminimality}
It is clear that a Horner System (for a given polynomial)
is not unique. Less obvious is the fact that a Horner system
is not necessarily a \emph{minimal} ALS.
This is shown in the following example:
Let $p = ab(xyz + yz + z + 1) + acxyz$. A \emph{Horner system}
for $p$ is given by the ALS $\als{A}$,
\begin{displaymath}
\begin{bmatrix}
1 & -a & . & . & . & . & . \\
. & 1 & -b & -c & . & . & . \\
. & . & 1 & -1 & -1 & -1 & -1  \\
. & . & . & 1 & -x & . & . \\
. & . & . & . & 1 & -y & . \\
. & . & . & . & . & 1 & -z \\
. & . & . & . & . & . & 1
\end{bmatrix}
s = 
\begin{bmatrix}
. \\ . \\ . \\ . \\ . \\ . \\ 1
\end{bmatrix},
\end{displaymath}
with $N(\als{A}) = N_s(\als{A}) = N_t(\als{A}) = 5$.
Adding column~3 to columns~4--7 and removing row~3 and column~3
yields the \emph{minimal} ALS $\als{A}'$,
\begin{displaymath}
\begin{bmatrix}
1 & -a & . & . & . & . \\
. & 1 & -b-c & -b & -b & -b \\
. & . & 1 & -x & . & . \\
. & . & . & 1 & -y & . \\
. & . & . & . & 1 & -z \\
. & . & . & . & . & 1
\end{bmatrix}
s = 
\begin{bmatrix}
. \\ . \\ . \\ . \\ . \\ 1
\end{bmatrix},
\end{displaymath}
with $N(\als{A}')=N_s(\als{A}')=6$ and $N_t(\als{A}')=7$.
However, evaluating $p$ using the representation as
(finite) formal sum ---in a naive way--- needs
13~multiplications. And since the worst case is of exponential
complexity, the restriction to \emph{minimal} admissible linear systems
will suffice in practice for a first ``evaluation simplification''.
See Table~\ref{tab:hs.count}.
\end{Remark}

\begin{proposition}[Evaluation Complexity of Polynomials]\label{pro:hs.complexity}
Let $p \in \freeALG{\field{K}}{\alphabet{X}}$ of rank $n \ge 2$.
Then $n-2 \le N(p) \le \frac{1}{2}(n-1)(n-2)$.
\end{proposition}

\begin{proof}
Since a polynomial of rank~$n$ admits a polynomial ALS
of dimension~$n$, the upper bound follows directly from
the upper unitriangular system matrix.
For the lower bound we can assume without loss of generality
that $N(p) = N_s(\als{A})$ for some polynomial ALS $\als{A}=(1,A,\lambda)$
of dimension $\dim \als{A} = m \ge n$.
However, $N_s(\als{A}) \le n-3$ would imply that there
are at least $m-n+1$ scalar columns in the system matrix $A$
(except column~1 which must not be touched and
column~$m$ which is irrelevant for $N_s$)
which could be removed after appropriate row operations,
contradicting that $n$ is the rank of $p$.
\end{proof}

\begin{Remark}\label{rem:hs.count}
The polynomial ALS from Table~\ref{tab:hs.count} for $p_k$ is
\begin{displaymath}
\begin{bmatrix}
1 & -(x+y+z) \\
  & 1 & -(x+y+z) \\
  &   & \ddots & \ddots \\
  &   &   & 1 & -(x+y+z) \\
  &   &   &   & 1
\end{bmatrix}
s = 
\begin{bmatrix}
0 \\ 0 \\ \vdots \\ 0 \\ 1
\end{bmatrix},
\end{displaymath}
that for $q_k$ is
\begin{displaymath}
\begin{bmatrix}
1 & -(x_1+y_1+z_1) & -(x_2+y_2+z_2) & \ldots & -(x_k+y_k+z_k) \\
  & 1 & -(x_1+y_1+z_1) & \ldots & -(x_{k-1} +y_{k-1}+z_{k-1}) \\
  &   & \ddots & \ddots & \vdots \\
  &   &   & 1 & -(x_1+y_1+z_1) \\
  &   &   &   & 1
\end{bmatrix}
s = 
\begin{bmatrix}
0 \\ 0 \\ \vdots \\ 0 \\ 1
\end{bmatrix}.
\end{displaymath}
Let $\ell$ be the number of letters in each entry of the system
matrix (here $\ell=3$) and $n$ the dimension of the (polynomial)
admissible linear system. Then there are $\ell$ words of length~1
in $s_{n-1}$ which we denote by ``$\ell\cdot 1$''.
In $s_{n-2}$ the words are $\ell\cdot 1$ and
$\ell$-times the words of $s_{n-1}$ with one additional letter,
that is, $\ell\cdot 1 + \ell^2 \cdot (1+1)$.
In $s_{n-3}$ the words are
  $\ell\cdot 1 + \ell^2 \cdot 2 + \ell\bigl(\ell \cdot (1+1) + \ell^2\cdot (2+1)\bigr)
  = \ell\cdot 1 + 2\,\ell^2 \cdot 2 + \ell^3\cdot 3$.
In $s_{n-4}$ and $s_{n-5}$ the words are
\begin{align*}
&\fbox{1}\,\ell\cdot 1 + \fbox{3}\,\ell^2\cdot 2 + \fbox{3}\,\ell^3\cdot 3
  + \fbox{1}\,\ell^4\cdot 4 \quad\text{resp.}\\
&\fbox{1}\,\ell\cdot 1 + \fbox{4}\,\ell^2\cdot 2 + \fbox{6}\,\ell^3\cdot 3
  + \fbox{4}\,\ell^4\cdot 4 + \fbox{1}\,\ell^5\cdot 5,
\end{align*}
revealing that the coefficients are the entries in the respective row
of the Pascal triangle
\begin{displaymath}
\begin{array}{ccccccccccc}
& & & & & 1 \\
& & & & 1 && 1 \\
& & & 1 && 2 && 1 \\
& & 1 && 3 && 3 && 1 \\
& 1 && 4 && 6 && 4 && 1 \\
1 && 5 && 10 && 10 && 5 && 1 \\
\revddots\ddots && \revddots\ddots && \revddots\ddots && \revddots\ddots && \revddots\ddots && \revddots\ddots 
\end{array}.
\end{displaymath}
Now both, the number of terms and the number of multiplications,
are immediate.
\end{Remark}

\section{Matrix Factorization}\label{sec:hs.mf}

Before we introduce the concept of the factorization of
polynomials into a product of matrices (aka ``matrix factorization'')
in Definition~\ref{def:hs.mred}
we recall the basics from the ``minimal'' multiplication of
polynomials and the opposite point of view, namely the
\emph{polynomial factorization}
(Theorem~\ref{thr:hs.factorization}).

\medskip
The factorization of polynomials into \emph{atoms}, that is,
\emph{irreducible elements} (Definition~\ref{def:hs.irrpoly})
corresponds to the transformation
of a \emph{minimal} polynomial admissible linear system
to one with a system matrix having the ``finest'' possible
upper right ``staircase'' of zeros, for example $p = xyz$
given by the ALS (Definition~\ref{def:hs.pals})
\begin{displaymath}
\begin{bmatrix}
1 & -x & 0 & 0 \\
. & 1 & -y & 0 \\
. & . & 1 & -z \\
. & . & . & 1
\end{bmatrix}
s =
\begin{bmatrix}
. \\ . \\ . \\ 1
\end{bmatrix},
\quad
s =
\begin{bmatrix}
xyz \\ yz \\ z \\ 1
\end{bmatrix}.
\end{displaymath}
Those upper right blocks of zeros come directly from
the ``minimal'' polynomial multiplication
\cite[Proposition~28]{Schrempf2017b9}
, illustrated in the following example.
Notice however, that this (upper right) form is not unique
in general.

\begin{Example}
Let $p = xy+1$ and $q = zx-3$ be given by the \emph{minimal} ALS
\begin{displaymath}
\begin{bmatrix}
1 & -x & -1 \\
. & 1 & -y \\
. & . & 1 
\end{bmatrix}
s =
\begin{bmatrix}
. \\ . \\ 1
\end{bmatrix}
\quad\text{and}\quad
\begin{bmatrix}
1 & -z & 3 \\
. & 1 & -x \\
. & . & 1
\end{bmatrix}
s =
\begin{bmatrix}
. \\ . \\ 1
\end{bmatrix}
\end{displaymath}
respectively. Recall that
the symbol $s$ for the solution vector is
used in a generic way.
By Proposition~\ref{pro:hs.ratop},
an ALS for the product $pq = (xy+1)(zx-3)$ is given by
\begin{displaymath}
\begin{bmatrix}
1 & -x & -1 & . & . & . \\
. & 1 & -y & . & . & . \\
. & . & 1 & -1 & . & . \\
. & . & . & 1 & -z & 3 \\
. & . & . & . & 1 & -x \\
. & . & . & . & . & 1
\end{bmatrix}
s = 
\begin{bmatrix}
. \\ . \\ . \\ . \\ . \\ 1
\end{bmatrix}.
\end{displaymath}
Now, if we add column~3 to column~4 we get
\begin{displaymath}
\begin{bmatrix}
1 & -x & -1 & -1 & . & . \\
. & 1 & -y & -y & . & . \\
. & . & 1 & 0 & . & . \\
. & . & . & 1 & -z & 3 \\
. & . & . & . & 1 & -x \\
. & . & . & . & . & 1
\end{bmatrix}
s = 
\begin{bmatrix}
. \\ . \\ 0 \\ . \\ . \\ 1
\end{bmatrix}
\end{displaymath}
where the third row equation reads $s_3 = 0$ and hence
we can remove row~3 and column~3 since there is no contribution
to the first component $s_1 = pq$ in the solution vector $s$.
Thus a \emph{minimal} ALS for $pq$ is given by
\begin{displaymath}
\begin{bmatrix}
1 & -x & -1 & 0 & 0 \\
. & 1 & -y & 0 & 0 \\
. & . & 1 & -z & 3 \\
. & . & . & 1 & -x \\
. & . & . & . & 1
\end{bmatrix}
s = 
\begin{bmatrix}
0 \\ 0 \\ . \\ . \\ 1
\end{bmatrix}.
\end{displaymath}
For concrete examples minimality can be checked easily
by using Proposition~\ref{pro:hs.cohn94.47}.
In the general case a systematic application of left
and right minimization steps (as in Algorithm~\ref{alg:hs.minals})
ensures minimality.
Notice the upper right $2 \times 2$ block of zeros
in the system matrix (and the upper zeros in the
right hand side).
\end{Example}

\begin{theorem}[Polynomial Factorization
\protect{\cite[Theorem~40]{Schrempf2017b9}
}]\label{thr:hs.factorization}
Let $p \in \freeALG{\field{K}}{\alphabet{X}}$ be given by the
\emph{minimal} polynomial ALS $\als{A} = (1,A,\lambda)$
of dimension $n = \rank p \ge 3$.
Then $p$ has a factorization into $p = q_1 q_2$ with $\rank(q_i) = n_i \ge 2$
if and only if there exists a polynomial transformation $(P,Q)$
such that $PAQ$ has an upper right block of zeros of size $(n_1 - 1) \times (n_2 - 1)$.
\end{theorem}

\begin{Example}
Let $p = 2aexc + 2bxc - aexd - bxd$
\cite{deOliveira2012a}
\ given by the \emph{minimal} ALS
$\als{A} = (u,A,v)$,
\begin{displaymath}
\begin{bmatrix}
1 & -a & -b & -a & . \\
. & 1 & -e & . & 2c-d \\
. & . & 1 & -x & . \\
. & . & . & 1 & d-2c \\
. & . & . & . & 1
\end{bmatrix}
s = 
\begin{bmatrix}
. \\ . \\ . \\ . \\ 1
\end{bmatrix}.
\end{displaymath}
To find a non-trivial factor of $p$ we need to find
a transformation $(P,Q)$ of the form
\begin{displaymath}
(P,Q) = \left(
\begin{bmatrix}
1 & \alpha_{1,2} & \alpha_{1,3} & \alpha_{1,4} & 0 \\
. & 1 & \alpha_{2,3} & \alpha_{2,4} & 0 \\
. & . & 1 & \alpha_{3,4} & 0 \\
. & . & . & 1 & 0 \\
. & . & . & . & 1
\end{bmatrix},
\begin{bmatrix}
1 & 0 & 0 & 0 & 0 \\
. & 1 & \beta_{2,3} & \beta_{2,4} & \beta_{2,5} \\
. & . & 1 & \beta_{3,4} & \beta_{3,5} \\
. & . & . & 1 & \beta_{3,5} \\
. & . & . & . & 1
\end{bmatrix}
\right)
\end{displaymath}
(see Definition~\ref{def:hs.pals})
such that $PAQ$ has an upper right block of
zeros of size $1 \times 3$, $2 \times 2$ or $3 \times 1$.
In this case it is (almost) immediate that
we need to add row~4 to row~2 and subtract column~2 from
column~4, that is, $\alpha_{2,4} = 1$
and $\beta_{2,4} = -1$,
\begin{displaymath}
(P,Q) = \left(
\begin{bmatrix}
1 & . & . & . & 0 \\
. & 1 & . & 1 & 0 \\
. & . & 1 & . & 0 \\
. & . & . & 1 & 0 \\
. & . & . & . & 1
\end{bmatrix},
\begin{bmatrix}
1 & 0 & 0 & 0 & 0 \\
. & 1 & . & -1 & . \\
. & . & 1 & . & . \\
. & . & . & 1 & . \\
. & . & . & . & 1
\end{bmatrix}
\right)
\end{displaymath}
and thus $P\als{A}Q = (uQ, PAQ, Pv)$,
\begin{displaymath}
\begin{bmatrix}
1 & -a & -b & 0 & 0 \\
. & 1 & -e & 0 & 0 \\
. & . & 1 & -x & 0 \\
. & . & . & 1 & d-2c \\
. & . & . & . & 1
\end{bmatrix}
s =
\begin{bmatrix}
. \\ . \\ . \\ . \\ 1
\end{bmatrix}.
\end{displaymath}
For the evaluation of $p$ (in the ``expanded'' form)
with $m \times m$ matrices, $10\,\complexity(m^3)$ operations
are necessary while only $3\,\complexity(m^3)$ suffice for
the factorized version
$p = 2aexc + 2bxc - aexd - bxd = (ae + b)x(2c-d)$.
\end{Example}

\begin{remark}
In general it can be difficult to find these (invertible)
transformation matrices (if they exist),
in particular, if the base field $\field{K}$ is not
algebraically closed, that is,
$\field{K}\subsetneq \aclo{\field{K}}$.
Testing (ir)reducibility works practically for rank $\le 12$,
in some cases up to rank $\le 17$
\cite[Chapter~2]{Janko2018a}
.
In the previous example it was easy because we can solve
a \emph{linear} system of equations for ``non-overlapping''
row and column transformations, that is,
if we use column~3 to create an upper right block of zeros
of size $2 \times 2$, we are not allowed to use row~3
(and vice versa). See also 
\cite[Remark~5.8]{Schrempf2018a2}
.
\end{remark}

\medskip
Before we formalize the factorization of a polynomial
into matrices we show the idea in an example.
A comprehensive theory for the work with matrices
(from an algebraic perspective including the
general factorization theory 
\cite{Schrempf2017c2}
) is considered in future work.
Here we need only the
fact that we can \emph{admissibly} transform a (polynomial)
ALS. If we find a certain pattern of zeros,
we can read off the matrices ---more or less--- directly
and their product yields the polynomial.
In the case of a polynomial matrix (not to be confused with
matrix polynomial), additional letters can be used to view it
as a ``classical'' nc polynomial
(Example~\ref{ex:hs.mv}).

\begin{Example}[``Matrix factorization'' of the Antikommutator]\label{ex:hs.anticomm}
Let $p = xy+yx$ given by the minimal polynomial ALS $\als{A} = (u,A,v)$,
\begin{displaymath}
\begin{bmatrix}
1 & - x & - y & 0 \\
. & 1 & 0 & -y \\
. & . & 1 & -x \\
. & . & . & 1
\end{bmatrix}
s = 
\begin{bmatrix}
. \\ . \\ . \\ 1
\end{bmatrix}.
\end{displaymath}
In this case only $2$~multiplications are necessary.
Notice the zeros in the system matrix. In this case we can write $p$
as a product of two matrices:
\begin{displaymath}
p =
\begin{bmatrix}
x & y
\end{bmatrix}
\begin{bmatrix}
1 & . \\
. & 1 
\end{bmatrix}\inv
\begin{bmatrix}
y \\ x
\end{bmatrix}
=
\begin{bmatrix}
x & y
\end{bmatrix}
\begin{bmatrix}
y \\ x
\end{bmatrix}.
\end{displaymath}
If $p$ is given by any other \emph{minimal} polynomial ALS
we can look for an admissible transformation $(P,Q)$ of the form
\begin{displaymath}
(P,Q) = 
\left(
\begin{bmatrix}
1 & \alpha_{1,2} & \alpha_{1,3} & 0 \\
. & 1 & \alpha_{2,3} & 0 \\
. & . & 1 & 0 \\
. & . & . & 1
\end{bmatrix},
\begin{bmatrix}
1 & 0 & 0 & 0 \\
. & 1 & \beta_{2,3} & \beta_{2,4} \\
. & . & 1 & \beta_{3,4} \\
. & . & . & 1
\end{bmatrix}
\right)
\end{displaymath}
such that $PAQ$ has the form (``$*$'' denotes some non-zero entry)
\begin{displaymath}
\begin{bmatrix}
1 & * & * & 0 \\
. & 1 & 0 & * \\
. & . & 1 & * \\
. & . & . & 1
\end{bmatrix}.
\end{displaymath}
This yields a \emph{non}-linear (polynomial) system of equations.
For details and how to solve such a systems we refer to
\cite[Section~4.4]{Schrempf2018c}
.
Notice that these transformation matrices
do \emph{not} suffice in general
because permutations of rows/columns are excluded.
Thus we need (admissible) transformations of the form
\begin{equation}\label{eqn:hs.gft}
(P,Q) = 
\left(
\begin{bmatrix}
\alpha_{1,1} & \ldots & \alpha_{1,n-1} & 0 \\
\vdots & \ddots & \vdots & \vdots \\
\alpha_{n-1,1} & \ldots & \alpha_{n-1,n-1} & 0 \\
\alpha_{n,1} & \ldots & \alpha_{n,n-1} & 1 \\
\end{bmatrix},
\begin{bmatrix}
1 & 0 & \ldots & 0 \\
\beta_{2,1} & \beta_{2,2} & \ldots & \beta_{2,n} \\
\vdots & \vdots & \ddots  & \vdots \\
\beta_{n,1} & \beta_{n,2} & \ldots & \beta_{n,n}
\end{bmatrix}
\right)
\end{equation}
and invertibility conditions $\det P \neq 0$
and $\det Q \neq 0$.
In such a case we call $(P,Q)$
(admissible) \emph{factorization transformation}.
\end{Example}

\begin{definition}[Matrix Reducibility]\label{def:hs.mred}
Let $p \in \freeALG{\field{K}}{\alphabet{X}}$ of rank $n \ge 3$
given by the \emph{minimal} polynomial ALS $\als{A} = (u,A,v)$
and $k \in \{ 1, 2, \ldots, n-2 \}$.
If there exists an $i \in \{ 1, 2, \ldots, n-k-1 \}$ and 
a factorization transformation $(P,Q)$ such that
$P\als{A} Q$ is again a polynomial ALS,
$PAQ$ has an upper right block of zeros of size $i \times (n-i-k)$
and an identity diagonal $k \times k$ block in rows $i+1$ to $i+k$,
that is, $PAQ$ has the form
\begin{displaymath}
\begin{array}{r}
i \text{ rows} \\
k \text{ rows} \\
n-i-k \text{ rows}
\end{array}
\begin{bmatrix}
* & * & 0 \\
0 & I_k & * \\
0 & 0 & *
\end{bmatrix},
\end{displaymath}
then $p$ is called $k$-reducible.
If there is no such $i$, it is called $k$-irreducible.
\end{definition}

\begin{remark}
$1$-irreducibility is just the ``classical'' irreducibility.
The anticommutator (Example~\ref{ex:hs.anticomm})
is ($1$-)irreducible but $2$-reducible.
\end{remark}

\begin{Example}
Let $p = 3 c y x b + 3 x b y x b + 2 c y x a x + c y b x b
       - c y a x b  - 2 x b y x a x + 4 x b y b x b  - 3 x b y a x b
       + 3 x a x y x b  - 3 b x b y x b + 6 a x b y x b + 2 x a x y x a x
       + x a x y b x b  - x a x y a x b  - 2 b x b y x a x  - b x b y b x b
       + b x b y a x b + 5 a x b y b x b - 4 a x b y a x b$
\cite[Section~8.2]{Camino2006a}
. The rank of $p$ is~$16$,
that is, the system matrix of a \emph{minimal} ALS has
dimension~16. The polynomial $p$ has 19~terms (monomials).
A \emph{minimal} (polynomial) ALS $\als{A}=(1,A,3)$
constructed iteratively starting with zero and
adding monomial by monomial (including minimization
by Algorithm~\ref{alg:hs.minals})
using the computer algebra system
\cite{FRICAS2019}
\ and the (experimental) implementation of the free field
\texttt{FDALG} ``Free Division ALGebra''
(available in Release~1.3.5)
results already in a very sparse ALS
showing that $p$ is 2-reducible for $i=7$
(in particular that $A$ has a upper right block of zeros
of size $7 \times 7$). The left number of multiplications
is $N_s(\als{A}) = 24$, the right number $N_t(\als{A}) = 22$.
For simplicity we show only the lower right
``subsystem'' of $\als{A}$ of size $9\times 9$
and hide the upper left
in the polynomials $p_1'$ and $p_2'$:
\begin{displaymath}
\begin{array}{r}
\\ 8 \\ 9 \\ 10 \\ 11 \\ 12 \\ 13 \\ 14 \\ 15 \\ 16
\end{array}
\begin{bmatrix}
1 & -p_1' & -p_2' & 0 & 0 & 0 & 0 & 0 & 0 & 0 \\
. & 1 & 0 & -3y & -4y & . & . & -y & . & . \\
. & 0 & 1 & -y & -y & . & . & -y & . & . \\
. & . & . & 1 & . & -5x & . & \frac{1}{3}a & . & . \\
. & . & . & . & 1 & 4x & . & -\frac{1}{3}b & . & . \\
. & . & . & . & . & 1 & -a & . & . & .  \\
. & . & . & . & . & . & 1 & . & . & -\frac{2}{3}x \\
. & . & . & . & . & . & . & 1 & -x & . \\
. & . & . & . & . & . & . & . & 1 & -b \\
. & . & . & . & . & . & . & . & . & 1 \\
\end{bmatrix}
s =
\begin{bmatrix}
 . \\ . \\ . \\ . \\ . \\ . \\ . \\ . \\ . \\ 3
\end{bmatrix}
\end{displaymath}
Only two elementary operations
(subtracting 3-times column~10 from column~14
and adding 2-times column~11 to column~14)
yield $\als{A}'$,
\begin{displaymath}
\begin{array}{r}
\\ 8 \\ 9 \\ 10 \\ 11 \\ 12 \\ 13 \\ 14 \\ 15 \\ 16
\end{array}
\begin{bmatrix}
1 & -p_1' & -p_2' & . & . & . & . & . & . & . \\
. & 1 & . & -3y & -4y & 0 & 0 & 0 & 0 & 0 \\
. & . & 1 & -y & -y   & 0 & 0 & 0 & 0 & 0 \\
. & . & . & 1 & 0 & -5x & . & \frac{1}{3}a & . & . \\
. & . & . & 0 & 1 & 4x & . & -\frac{1}{3}b & . & . \\
. & . & . & . & . & 1 & -a & . & . & .  \\
. & . & . & . & . & . & 1 & . & . & -\frac{2}{3}x \\
. & . & . & . & . & . & . & 1 & -x & . \\
. & . & . & . & . & . & . & . & 1 & -b \\
. & . & . & . & . & . & . & . & . & 1 \\
\end{bmatrix}
s =
\begin{bmatrix}
 . \\ . \\ . \\ . \\ . \\ . \\ . \\ . \\ . \\ 3
\end{bmatrix},
\end{displaymath}
revealing that $p$ is also 2-reducible for $i=9$
and (by Proposition~\ref{pro:hs.complexity})
$14 \le N(p) \le 20 = N_t(\als{A}')$.
Recursively, by using an ALS as a ``workbench'',
one can find the matrix factorization
$p = (X_1 X_2 X_3 + X_4) Y Z_1 Z_2 Z_3$ with
\begin{displaymath}
X_1 = 
\begin{bmatrix}
1+a & 1+b & x
\end{bmatrix},\quad
X_2 = 
\begin{bmatrix}
x & . & . \\
. & x & . \\
. & . & a 
\end{bmatrix},\quad
X_3 = 
\begin{bmatrix}
b & . \\
. & -b \\
. & x 
\end{bmatrix},\quad
X_4 = 
\begin{bmatrix}
. & c 
\end{bmatrix},
\end{displaymath}
\begin{displaymath}
Y =
\begin{bmatrix}
y & . \\
. & y 
\end{bmatrix},\quad
Z_1 = 
\begin{bmatrix}
6+5b-4a & . \\
3+b-a & 2x 
\end{bmatrix},\quad
Z_2 = 
\begin{bmatrix}
. & x \\
a & . 
\end{bmatrix},\quad
\text{and}\quad
Z_3 =
\begin{bmatrix}
x \\ b
\end{bmatrix},
\end{displaymath}
showing that $N(p) = 15$.
A ``block'' polynomial ALS for $p$ is
\begin{displaymath}
\begin{bmatrix}
1 & -X_1 & . & -X_4 & . & . & . & . \\
. & I_3 & -X_2 & . & . & . & . & . \\
. & . & I_3 & -X_3 & . & . & . & . \\
. & . & . & I_2 & -Y & . & . & . \\
. & . & . & . & I_2 & -Z_1 & . & . \\
. & . & . & . & . & I_2 & -Z_2 & . \\
. & . & . & . & . & . & I_2 & -Z_3 \\
. & . & . & . & . & . & . & 1
\end{bmatrix}
s = 
\begin{bmatrix}
. \\ . \\ . \\ . \\ . \\ . \\ . \\ 1
\end{bmatrix}.
\end{displaymath}
Notice that $p$ is (1-)\emph{irreducible}
because it is not possible to (admissibly) transform a
minimal ALS (for $p$) into one with an upper right block
of zeros of size $1 \times 14$, $2 \times 13$, \ldots, $13 \times 2$
or $14 \times 1$.
\end{Example}

\begin{remark}
To evaluate $p$ as (finite) formal sum,
97~matrix-matrix multiplications are necessary.
On the other hand $N(p)=15$, that is,
only 15~multiplications (starting from the top left)
are necessary using a Horner system.
\end{remark}

\begin{remark}
The matrix factorization $p=X Y Z$ shows immediately that
the \emph{Sylvester index} \cite{Konstantinov2000a}
\ (with respect to $y$) is~2, that is,
$p = p_1 y q_1 + p_2 y q_2$.
\end{remark}

\begin{Example}\label{ex:hs.mv}
Taking the polynomial matrix
\begin{displaymath}
X =
\begin{bmatrix}
(ae+b)x_{11}(2c-d) & (ae+b)x_{12}(c+d) - (ae+b)x_{11} d \\
\bigl( b x_{21} - (ae+b) x_{11}\bigr) (2c-d) &
  (ae+b)\bigl(x_{11}d - x_{12}(c+d) \bigr) - b x_{21} d + b x_{22}(c+d)
\end{bmatrix}
\end{displaymath}
from \cite{deOliveira2012a}
, by multiplying a row vector from the left
respectively a column vector from the right
(both with generic variables), we can consider it
as a polynomial:
\begin{displaymath}
p = 
\begin{bmatrix}
y_1 & y_2 
\end{bmatrix}
X
\begin{bmatrix}
z_1 \\ z_2
\end{bmatrix}.
\end{displaymath}
A \emph{minimal} ALS for $p$ is
\begin{displaymath}
\left[ \begin{array}{ccccccccccc}
\tabstrut
1 & -y_1 & -y_2 & .  & .  & .  & .  & .  & .  & . & . \\
  &  1   & .    & -a  & -b & 0  & .  & .  & .  & . & . \\
  &      &  1   & a & b  & -b & .  & .  & .  & . & . \\
  &      &      & 1  & -e  & 0  & .  & .  & .  & . & . \\
  &      &      &    & 1  & .  & -x_{11} & -x_{12} & .  & . & . \\
  &      &      &    &    & 1  & -x_{21} & -x_{22} & .  & . & . \\
  &      &      &    &    &    & 1  & .  & d-2c & d & . \\
  &      &      &    &    &    &    & 1  & 0  & -d-c & . \\
  &      &      &    &    &    &    &    & 1  & . & -z_1 \\
  &      &      &    &    &    &    &    &    & 1 & -z_2 \\
  &      &      &    &    &    &    &    &    &   & 1 
\end{array}  \right]
s = 
\left[ \begin{array}{c}
\tabstrut
. \\ . \\ . \\ . \\ . \\ . \\ . \\ . \\ . \\ . \\ 1 
\end{array} \right]
\end{displaymath}
which translates directly ---the second matrix
appears in a linearized form in the system matrix---
into the matrix factorization
\begin{displaymath}
p =
\begin{bmatrix}
y_1 & y_2
\end{bmatrix}
\underbrace{%
\begin{bmatrix}
ae+b & 0 \\
-ae-b & b
\end{bmatrix}
\begin{bmatrix}
x_{11} & x_{12} \\
x_{21} & x_{22}
\end{bmatrix}
\begin{bmatrix}
2c-d & -d \\
0 & c+d
\end{bmatrix}}_{=X}
\begin{bmatrix}
z_1 \\ z_2
\end{bmatrix}.
\end{displaymath}
Notice that we already know that there
must exist an upper right block of zeros of size $1 \times 8$
and one of size $8 \times 1$.
The tricky part here is that one ``matrix-factor'' of $X$ has
\emph{non-linear} entries.
In this case one could substitute $ae$ by
a new symbol/letter. Then the ALS would have
dimension~10 (recall that $\rank p = 11$).
\end{Example}

\section{Epilogue}\label{sec:hs.epi}

Given a (non-trivial) nc polynomial $p$ with $\rank p = n \ge 2$,
a Horner system is the most sparse polynomial admissible linear system
for $p$ with respect to matrix-matrix multiplications.
Unfortunately, finding Horner systems in general is
very difficult since one needs to solve non-linear
(polynomial) systems of equations.
However, in concrete situations,
one can get good ``approximations'' quite easily
by starting with a \emph{minimal} polynomial ALS $\als{A}$
(constructed by Algorithm~\ref{alg:hs.minals})
and trying to find non-trivial factorizations
by linear techniques 
\cite[Remark~5.8]{Schrempf2018a2}
, that is, using ``non-overlapping'' row and column
transformations, yielding some ALS $\als{A}'$.
From Proposition~\ref{pro:hs.complexity} we have bounds for
the minimal number of multiplications (for the evaluation of $p$),
namely,
\begin{displaymath}
n-2 \le N(p) \le N(\als{A}') \le \tsfrac{1}{2}(n-1)(n-2).
\end{displaymath}
If $n-2 \ll N(\als{A}') \le \tsfrac{1}{2}(n-1)(n-2)$
it might be worth to check systematically for
$k$-reducibility of $p$ for $k=1,2,\ldots,n-2$
(Definition~\ref{def:hs.mred}).
This can be done recursively, using already known
factorizations in $\als{A}'$,
yielding some $\als{A}''$.
And finally one can minimize the number of non-scalar
entries in the ``matrix factors'' of $\als{A}''$
by looking for appropriate (scalar) invertible matrices,
for example, $P \in \field{K}^{k_1 \times k_1}$
and $Q \in \field{K}^{k_2 \times k_2}$,
\begin{displaymath}
p = \underbrace{\tabstrut X P}_{=:X'} \,
    \underbrace{\tabstrut P\inv Y Q}_{=:Y'} \,
    \underbrace{\tabstrut Q\inv Z}_{=:Z'}.
\end{displaymath}
In general, this is very difficult, since already
for a special case, namely \emph{pivot block refinement}
\cite[Section~3]{Schrempf2018a2}
, one needs to solve non-linear (polynomial) systems of equations.
(See also 
\cite[Section~4.4]{Schrempf2018c}
.)
If there is no additional structure one can use,
this is comparable to testing ``fullness'' of matrices
\cite[Chapter~3]{Janko2018a}
, so one cannot expect a brute-force approach
to work practically for $k_1 = k_2 = k > 5$.

\medskip
For ``totally'' irreducible elements of rank~$n$ one would need to
check for all ``sparsity patterns'' with respect to evaluation by
the left and by the right family.
For an ALS $\als{A} = (1,A,\lambda)$
of dimension~$n$ there are $\bar{n} := (n-2)(n-1)/2-1$ entries
to test for $n-3$, $n-2$, \ldots, $\bar{n}-1$ non-scalar entries.
For illustration we take $n=5$:
\begin{displaymath}
\underbrace{%
\begin{bmatrix}
\alpha_{1,1} & \alpha_{1,2} & \alpha_{1,3} & \alpha_{1,4} & 0 \\
\alpha_{2,1} & \alpha_{2,2} & \alpha_{2,3} & \alpha_{2,4} & 0 \\
\alpha_{3,1} & \alpha_{3,2} & \alpha_{3,3} & \alpha_{3,4} & 0 \\
\alpha_{4,1} & \alpha_{4,2} & \alpha_{4,3} & \alpha_{4,4} & 0 \\
\alpha_{5,1} & \alpha_{5,2} & \alpha_{5,3} & \alpha_{5,4} & 1 \\
\end{bmatrix}}_{=P \in \field{K}^{n\times n},\, \det P \neq 0}
\underbrace{%
\begin{bmatrix}
1 & * & ? & ? & * \\
. & 1 & ? & ? & * \\
. & . & 1 & ? & * \\
. & . & . & 1 & * \\
. & . & . & . & 1
\end{bmatrix}}_{=A}
\underbrace{%
\begin{bmatrix}
1 & 0 & 0 & 0 & 0 \\
\beta_{2,1} & \beta_{2,2} & \beta_{2,3} & \beta_{2,4} & \beta_{2,5} \\
\beta_{3,1} & \beta_{3,2} & \beta_{3,3} & \beta_{3,4} & \beta_{3,5} \\
\beta_{4,1} & \beta_{4,2} & \beta_{4,3} & \beta_{4,4} & \beta_{4,5} \\
\beta_{5,1} & \beta_{5,2} & \beta_{5,3} & \beta_{5,4} & \beta_{5,5} \\
\end{bmatrix}}_{=Q \in \field{K}^{n\times n},\, \det Q \neq 0}
\end{displaymath}
This would yield $10+10+5 = 25$ possibilities
already for $n=5$, each inducing a \emph{non}-linear (polynomial)
system of equations with $2n(n-1)$ \emph{commuting} unknowns.
So in general this will not be very useful, in particular
because ---compared to the factorization---
one does not get any ``structural'' insight.
However, heuristic approaches for increasing (non-scalar) sparsity,
that is, ``approximating'' Horner systems,
by ``local'' row and column transformations depending
on the existing structure in the system matrix
(respectively the coefficient matrices $A_1,\ldots,A_d$
of the linear matrix pencil $A$) might be possible
and could be very helpful.

\begin{Remark}
The evaluation of a polynomial $p$ given by the
minimal polynomial ALS $\als{A}=(1,A,\lambda)$ of dimension~$n$
by non-square matrices (of appropriate size)
yields a natural block structure, entry $(i,j)$ in
the system matrix $A$ has size $m_i \times m_j$
with $m_1 = m_n = 1$.
In this case one can get a priority for checking
in particular \hbox{$1$-reducibility} to avoid huge inner dimensions,
for example in $p = x_1 x_2 x_3 x_4$ with
row vectors $x_1,x_3$ and column vectors $x_2,x_4$.
Here the factorization $p = (x_1 x_2)(x_3 x_4)$ is
of higher importance with respect to the evaluation.
\end{Remark}

For further references with respect to the application of nc polynomials
(and appropriate software for symbolic computations) we
refer to \cite{Camino2006a}
\ and \cite{deOliveira2012a}
. There is a close connection to
optimization respectively 
\emph{semidefinite programming} (SDP)
\cite{Blekherman2013a}
, in particular visible in
\cite{Cafuta2011a}
.
As a starting point for the evaluation of
\emph{commutative} multivariate polynomials
one could take
\cite{Czekansky2015a}
. If one has huge arrays of \emph{commutative} polynomials
to evaluate it might be possible to use non-commutativity
(in terms of matrix-matrix multiplication) like in
\cite{Dym2013a}
.

\section*{Acknowledgement}

This work has been partially supported by research subsidies granted by
the government of Upper Austria
(research project ``Methodenentwicklung für Energie\-fluss\-opti\-mie\-rung'').

\bibliographystyle{alpha}
\bibliography{doku}

\begin{thebibliography}{{Sch}18b}

\bibitem[Ami66]{Amitsur1966a}
S.~A. Amitsur.
\newblock Rational identities and applications to algebra and geometry.
\newblock {\em J. Algebra}, 3:304--359, 1966.

\bibitem[BGKR08]{Bart2008a}
H.~Bart, I.~Gohberg, M.~A. Kaashoek, and A.~C.~M. Ran.
\newblock {\em Factorization of matrix and operator functions: the state space
  method}, volume 178 of {\em Operator Theory: Advances and Applications}.
\newblock Birkh\"auser Verlag, Basel, 2008.
\newblock Linear Operators and Linear Systems.

\bibitem[BPT13]{Blekherman2013a}
G.~Blekherman, P.~A. Parrilo, and R.~R. Thomas, editors.
\newblock {\em Semidefinite optimization and convex algebraic geometry},
  volume~13 of {\em MOS-SIAM Series on Optimization}.
\newblock Society for Industrial and Applied Mathematics (SIAM), Philadelphia,
  PA; Mathematical Optimization Society, Philadelphia, PA, 2013.

\bibitem[BR11]{Berstel2011a}
J.~Berstel and C.~Reutenauer.
\newblock {\em Noncommutative rational series with applications}, volume 137 of
  {\em Encyclopedia of Mathematics and its Applications}.
\newblock Cambridge University Press, Cambridge, 2011.

\bibitem[BS15]{Baeth2015a}
N.~R. Baeth and D.~Smertnig.
\newblock Factorization theory: from commutative to noncommutative settings.
\newblock {\em J. Algebra}, 441:475--551, 2015.

\bibitem[CHS06]{Camino2006a}
J.~F. Camino, J.~W. Helton, and R.~E. Skelton.
\newblock Solving matrix inequalities whose unknowns are matrices.
\newblock {\em SIAM J. Optim.}, 17(1):1--36, 2006.

\bibitem[CKP11]{Cafuta2011a}
K.~Cafuta, I.~Klep, and J.~Povh.
\newblock N{CSOS}tools: a computer algebra system for symbolic and numerical
  computation with noncommutative polynomials.
\newblock {\em Optim. Methods Softw.}, 26(3):363--380, 2011.

\bibitem[Coh63]{Cohn1963b}
P.~M. Cohn.
\newblock Noncommutative unique factorization domains.
\newblock {\em Trans. Amer. Math. Soc.}, 109:313--331, 1963.

\bibitem[Coh74]{Cohn1974a}
P.~M. Cohn.
\newblock Progress in free associative algebras.
\newblock {\em Israel J. Math.}, 19:109--151, 1974.

\bibitem[Coh75]{Cohn1975b}
P.~M. Cohn.
\newblock Algebra and language theory.
\newblock {\em Bull. London Math. Soc.}, 7:1--29, 1975.

\bibitem[Coh95]{Cohn1995a}
P.~M. Cohn.
\newblock {\em Skew fields}, volume~57 of {\em Encyclopedia of Mathematics and
  its Applications}.
\newblock Cambridge University Press, Cambridge, 1995.
\newblock Theory of general division rings.

\bibitem[Coh06]{Cohn2006a}
P.~M. Cohn.
\newblock {\em Free ideal rings and localization in general rings}, volume~3 of
  {\em New Mathematical Monographs}.
\newblock Cambridge University Press, Cambridge, 2006.

\bibitem[CR94]{Cohn1994a}
P.~M. Cohn and C.~Reutenauer.
\newblock A normal form in free fields.
\newblock {\em Canad. J. Math.}, 46(3):517--531, 1994.

\bibitem[CR99]{Cohn1999a}
P.~M. Cohn and C.~Reutenauer.
\newblock On the construction of the free field.
\newblock {\em Internat. J. Algebra Comput.}, 9(3-4):307--323, 1999.
\newblock Dedicated to the memory of Marcel-Paul Sch{\"u}tzenberger.

\bibitem[CS15]{Czekansky2015a}
J.~Czekansky and T.~Sauer.
\newblock The multivariate {H}orner scheme revisited.
\newblock {\em BIT}, 55(4):1043--1056, 2015.

\bibitem[DDHK07]{Demmel2007a}
J.~Demmel, I.~Dumitriu, O.~Holtz, and R.~Kleinberg.
\newblock Fast matrix multiplication is stable.
\newblock {\em Numer. Math.}, 106(2):199--224, 2007.

\bibitem[DHM13]{Dym2013a}
H.~Dym, J.~W. Helton, and C.~Meier.
\newblock Non-commutative representations of families of {$k^2$} commutative
  polynomials in {$2k^2$} commuting variables.
\newblock {\em Internat. J. Algebra Comput.}, 23(7):1685--1753, 2013.

\bibitem[dO12]{deOliveira2012a}
M.~de~Oliveira.
\newblock Simplification of symbolic polynomials on non-commutative variables.
\newblock {\em Linear Algebra Appl.}, 437(7):1734--1748, 2012.

\bibitem[Fri19]{FRICAS2019}
{\em \textsc{FriCAS} Computer Algebra System}, 2019.
\newblock W.~Hebisch,\\ \url{http://axiom-wiki.newsynthesis.org/FrontPage}.

\bibitem[{Gan}66]{Gantmacher1966a}
F.~R. {Gantmacher}.
\newblock {\em {Theorie der Matrizen. 2., erg. Aufl.}}
\newblock {Moskau: Verlag 'Nauka'. Hauptredaktion f\"ur
  physikalisch-mathematische Literatur. 576 S.}, 1966.

\bibitem[Hig08]{Higham2008a}
N.~J. Higham.
\newblock {\em Functions of matrices}.
\newblock Society for Industrial and Applied Mathematics (SIAM), Philadelphia,
  PA, 2008.
\newblock Theory and computation.

\bibitem[Jan18]{Janko2018a}
B.~Janko.
\newblock {F}actorization of non-commutative {P}olynomials and {T}esting
  {F}ullness of {M}atrices.
\newblock Diplomarbeit, TU Graz, 2018.

\bibitem[KMP00]{Konstantinov2000a}
M.~Konstantinov, V.~Mehrmann, and P.~Petkov.
\newblock On properties of {S}ylvester and {L}yapunov operators.
\newblock {\em Linear Algebra Appl.}, 312(1-3):35--71, 2000.

\bibitem[{Sch}17]{Schrempf2017c2}
K.~{Schrempf}.
\newblock A factorization theory for some free fields.
\newblock {\em arXiv e-prints}, December 2017.
\newblock Version 2, March 2019, \url{http://arxiv.org/pdf/1712.09102}.

\bibitem[{Sch}18a]{Schrempf2018a2}
K.~{Schrempf}.
\newblock {A Standard Form in (some) Free Fields: How to construct Minimal
  Linear Representations}.
\newblock {\em arXiv e-prints}, March 2018.
\newblock Version 2, March 2019, \url{http://arxiv.org/pdf/1803.10627}.

\bibitem[{Sch}18b]{Schrempf2018c}
K.~{Schrempf}.
\newblock Free fractions: An invitation to (applied) free fields.
\newblock {\em ArXiv e-prints}, September 2018.

\bibitem[Sch18c]{Schrempf2017a9}
K.~Schrempf.
\newblock Linearizing the word problem in (some) free fields.
\newblock {\em Internat. J. Algebra Comput.}, 28(7):1209--1230, 2018.

\bibitem[Sch19]{Schrempf2017b9}
K.~Schrempf.
\newblock On the factorization of non-commutative polynomials (in free
  associative algebras).
\newblock {\em Journal of Symbolic Computation}, 94:126--148, 2019.

\bibitem[Sim16]{Simoncini2016a}
V.~Simoncini.
\newblock Computational methods for linear matrix equations.
\newblock {\em SIAM Rev.}, 58(3):377--441, 2016.

\bibitem[SS78]{Salomaa1978a}
A.~Salomaa and M.~Soittola.
\newblock {\em Automata-theoretic aspects of formal power series}.
\newblock Springer-Verlag, New York-Heidelberg, 1978.
\newblock Texts and Monographs in Computer Science.

\bibitem[Str69]{Strassen1969a}
V.~Strassen.
\newblock Gaussian elimination is not optimal.
\newblock {\em Numer. Math.}, 13:354--356, 1969.

\bibitem[TOT14]{Tajima2014a}
S.~Tajima, K.~Ohara, and A.~Terui.
\newblock An extension and efficient calculation of the {H}orner's rule for
  matrices.
\newblock In {\em Mathematical software---{ICMS} 2014}, volume 8592 of {\em
  Lecture Notes in Comput. Sci.}, pages 346--351. Springer, Heidelberg, 2014.

\end{thebibliography}
\addcontentsline{toc}{section}{Bibliography}

\end{document}